\documentclass[12pt,reqno]{amsart}
\usepackage[margin=1in]{geometry}
\usepackage{latexsym}
\usepackage{amssymb,latexsym,amscd,amsmath,amsthm,mathtools}
\usepackage{mathrsfs}
\usepackage{url}

\newtheorem{thm}{Theorem}

\newtheorem{lem}{Lemma}

\def\C{\mathbb{C}}
\def\R{\mathbb{R}}

\newcommand{\be}{\begin{equation}}
\newcommand{\ee}{\end{equation}}

\newcommand{\cF}{\mathcal{F}}
\newcommand{\cE}{\mathcal{E}}
\newcommand{\cQ}{\mathcal{Q}}
\newcommand{\re}{\mathop{\rm Re}}


\begin{document}
\title{Combinations of $L$-functions and their Non-coincident Zeros for $\sigma>1$}
\author{Scott Kirila}
\date{\today}
\address{Department of Mathematics, University of Exeter, Exeter, EX4 4QF, UK}
\email{s.kirila@exeter.ac.uk}
\maketitle

\section{Introduction}
\begin{abstract}
The purpose of this note is to build upon work of Booker--Thorne and Righetti concerning zeros of algebraic combinations of $L$-functions. Namely, we show that two generic combinations of functions from a wide class of Euler products have non-coincident zeros in the half-plane $\sigma>1$.
\end{abstract}

In 1936, Davenport--Heilbronn~\cite{DaHe} showed that Hurwitz zeta functions $\zeta(s,\alpha)$ with positive rational parameter $\alpha\ne \frac12$ or $1$ have infinitely many zeros to the right of the line $\re(s)=1$. This provided an infinite family of functions which vanish in the half-plane of absolute convergence, despite being a linear combination of Dirichlet $L$-functions, which do not vanish in this region; in particular, such linear combinations fail to obey a Riemann hypothesis. Since then, many authors have expanded upon Davenport and Heilbronn's work, such as Saias--Weingartner~\cite{SaWe}, who proved a comparable result for general linear combinations of Dirichlet $L$-functions. Their method, which relies on a type of weak universality property, has since been refined by Booker--Thorne~\cite{BT} to include algebraic combinations of $L$-functions associated to automorphic forms (assuming the generalized Ramanujan conjecture), and by Righetti~\cite{R} to a larger class of Euler products.

It is generally expected that two different primitive $L$-functions have (almost) no coincident zeros. In this direction, Fujii~\cite{Fu} showed that two primitive Dirichlet $L$-functions have a positive proportion of zeros which are non-coincident. One might expect that the same is typically true of two combinations of $L$-functions, assuming we avoid trivial obstructions such as having a common factor. The purpose of this note is to explore this question in the region $\sigma>1$. To formally state our result, we denote by $\mathcal{E}$ be the class of Dirichlet series appearing in \cite{R}. Namely, it consists of functions $F:\C\to \C$ satisfying the following axioms:
    \begin{enumerate}
    \setlength\itemsep{1em}
        \item[\textbf{(E1)}] the function $F$ is given by a Dirichlet series, $F(s)=\sum_{n=1}^{\infty}\frac{a_F(n)}{n^s}$, which converges absolutely for $\sigma>1$;
        
        \item[\textbf{(E2)}] it has an Euler product, so that $\log F(s)=\sum_pF_p(s)=\sum_p\sum_{k=1}^{\infty}\frac{b_F(p^k)}{p^{ks}}$, which also converges absolutely for $\sigma>1$;
        
        \item[\textbf{(E3)}] there is a positive constant $K_F$ such that $|a_F(p)|\le K_F$ for all primes $p$;
        
        \item[\textbf{(E4)}] the sum $\sum_p\sum_{k=2}^{\infty}\frac{|b_F(p^k)|}{p^k}$ is finite;
        
        \item[\textbf{(E5)}] for any $F,G\in \cE$, there is a complex number $m_{F,G}$ such that
            \[
            \sum_{p\le x}\frac{a_F(p)\overline{a_G(p)}}{p}=(m_{F,G}+o(1))\log\log x
            \]
            with $m_{F,F}>0$. We say that $F,G\in \cE$ are \emph{orthogonal} if $m_{F,G}=0$.
    \end{enumerate}
Examples of functions in this class are provided in \cite[Section~1.1]{R} and conjecturally include all $L$-functions. 
Righetti notes that $\cE$ likely contains other types of functions as well, since elements are neither required to satisfy a functional equation nor to possess a meromorphic continuation to the rest of the plane. An example of such is given by the Euler product
    \[
    \mathcal{Z}(s):=\prod_{n=1}^{\infty}\left(1-\frac{1}{p_{2n}^s}\right)^{-1};
    \]
as usual, we write $p_k$ to denote the $k$-th prime when ordered by increasing size. It follows from work of Grosswald--Schnitzer\cite{GrSc} that its square, $\mathcal{Z}(s)^2$, may be meromorphically continued to the half-plane $\sigma>0$ with a simple pole at $s=1$ and precisely the same zeros as the Riemann zeta function. Thus, although $\mathcal{Z}(s)$ is in $\cE$, it also has branch points at $s=1$ and any non-trivial zeros (with odd multiplicity) of the zeta function.

The polynomials we use for our algebraic combinations take their coefficients from the ring of $p$-finite Dirichlet series which converge absolutely in the half-plane $\sigma>1$, namely
    \[
    \cF:=\left\{\sum_{n\in \langle\cQ\rangle}\frac{a(n)}{n^s}\text{ abs. conv. for }\sigma\ge \alpha: \cQ\subset\mathcal{P}\text{ has finitely many elements}\right\}.
    \]
for example, $\cF$ contains all non-vanishing finite Euler products and their inverses, as well as any Dirichlet polynomial. Our main result is the following.

\begin{thm}\label{thm: >1}
	Let $f$ and $g$ be relatively prime polynomials over $\cF$ in $N_f$ and $N_g$ many variables, respectively, and assume that neither $f$ nor $g$ is a monomial. Suppose $F_1,\ldots,F_{N_f}\in\cE$ are pairwise orthogonal, and similarly for $G_1,\ldots,G_{N_g}\in \cE$. Then there is a zero of $f(F_1(s),\ldots,F_{N_f}(s))$ to the right of $\sigma=1$ which is not a zero of $g(G_1(s),\ldots,G_{N_g}(s))$, and vice versa.
	
\end{thm}
\noindent Once we have one non-coincident zero, we are guaranteed by almost-periodicity the existence of $\gg T$ many up to height $T$ in any vertical strip containing this point. There can be at most $\ll T$ zeros in any closed strip in $\sigma>1$, and so we find that a positive proportion of zeros in such regions are non-coincident. It is also worth noting that our proof shows that there are non-coincident zeros arbitrarily close to the line $\sigma=1$. Lastly, we briefly remark that an analogue of Theorem~\ref{thm: >1} can also be proved in closed strips contained in the region $\frac12<\sigma<1$ using a joint universality result due to Lee--Nakamura--Pa\'nkowski~\cite{LNP} for the Selberg class, assuming a quantitative form of Selberg's orthogonality conjecture\footnote{That is, a quantitative refinement of (E5).}. 

\subsection{A couple of examples}
\subsubsection*{Hurwitz zeta functions}
Suppose $1\le a\le q$ with $(a,q)=1$ and $q\ge 3$. The Hurwitz zeta function $\zeta(s,a/q)$ is defined as
	\[
	\zeta(s,a/q)=\sum_{n=0}^{\infty}(n+a/q)^{-s}
	\]
for $\sigma>1$, and it has a meromorphic continuation to the rest of the complex plane. It may also be expressed as a linear combination of Dirichlet $L$-functions, that is,
	\[
	\zeta(s,a/q)=\frac{q^s}{\varphi(q)}\sum_{\chi\bmod{q}}\overline{\chi}(a)L(s,\chi);
	\]
here $\varphi$ is the Euler totient function and the sum is over all Dirichlet characters modulo $q$. As mentioned in the beginning of the paper, Davenport and Heilbronn~\cite{DaHe} showed that $\zeta(s,a/q)$ has infinitely many zeros to the right of $\sigma=1$ if $q\ge 3$; that is, if $\zeta(s,a/q)$ is not simply a multiple of the Riemann zeta function. Now, the Dirichlet $L$-functions modulo $q$ are pairwise orthogonal functions in $\mathcal{E}$. Two linear polynomials in $\C[x_1,\ldots,x_{\varphi(q)}]$ are relatively prime unless one is a constant multiple of the other, which, for the polynomials $\sum_{\chi\bmod q_1}\overline{\chi}(a_1)x_{\chi}$ and $\sum_{\psi\bmod q_2}\overline{\psi}(a_2)x_{\psi}$, would require that $q_1=q_2$ and $\chi(a_1)\overline{\chi}(a_2)=c$ for all $\chi$ and some non-zero complex number $c$. Summing both sides of this last equation shows that we must have $c=1$ and $a_1=a_2$. Therefore, we may apply Theorem~\ref{thm: >1} to $\zeta(s,a_1/q_1)$ and $\zeta(s,a_2/q_2)$ provided that $q_1\ne q_2$ or $a_1\ne a_2$. That is, two distinct Hurwitz zeta functions (with rational parameters) have many non-coincident zeros in the half-plane $\sigma>1$ if $q\ge 3$.

\subsubsection*{$L$-functions associated to cuspforms}
The following can be viewed as a companion to \cite[Theorem~1.1]{BT}. Let $f_1,f_2\in S_k(\Gamma_1(N))$ be two holomorphic cuspforms and let $\Lambda_{f_j}(s)=\int_0^{\infty}f_j(iy)y^{s-1}dy$ be the associated complete $L$-functions, $j=1,2$. Here the region of absolute convergence is $\sigma>\frac{k+1}{2}$, as we have chosen not to normalize the Fourier coefficients for the sake of comparison. Booker and Thorne showed that the non-vanishing of $\Lambda_{f_j}(s)$ in the region $\sigma>\frac{k+1}{2}$ implies that $f_j$ must be an eigenfunction of the Hecke operators $T_p$ for all primes $p$ not dividing $N$. Instead, if we assume that $\Lambda_{f_1}(s)$ and $\Lambda_{f_2}(s)$ vanish at precisely the same non-empty set of points in the half-plane $\sigma>\frac{k+1}{2}$, then our Theorem~\ref{thm: >1} implies that $f_1=cf_2$ for some nonzero complex number $c$.

\section{Some lemmas}
As in \cite{BT}\cite{R}\cite{SaWe} for a single combination, the proof of Theorem~\ref{thm: >1} relies on a weak universality result in the region $\sigma>1$; below we state \cite[Prop.~1]{R}, which appears as our Lemma~\ref{prop: R}. Meanwhile, Lemma~\ref{lem: bt 1} produces an appropriate zero of $f$ at which $g$ does not vanish, while Lemma~\ref{lem: bt 2} helps to bridge the gap from working with $f(x)$ and $g(x)$ to $f(F_1(s),\ldots,F_{N_f}(s))$ and $g(G_1(s),\ldots,G_{N_g}(s))$.

\begin{lem}[Righetti]\label{prop: R}
Suppose $F_1,\ldots,F_N$ are distinct functions in $\mathcal{E}$. Then, for $y,R\ge 1$, there exists $\eta>0$ such that
	\begin{multline*}
	\left\{\left(\prod_{p>y}F_{j,p}(\sigma+it_p)\right)_{1\le j\le N}:t_p\in \R \text{ for each prime }p>y\right\}\\
	\supset \{(z_1,\ldots,z_N)\in\C^N:R^{-1}\le |z_j|\le R\text{ for each }1\le j\le N\}
	\end{multline*}
for every $\sigma \in(1,1+\eta]$.
\end{lem}

Next we have a natural extension of \cite[Lemma~2.4]{BT} to two polynomials.
\begin{lem}\label{lem: bt 1}
Let $f,g\in \C [x_1,\ldots,x_N]$ be relatively prime. Then either $f$ is a monomial, or there is an $x\in \C^N$ with $f(x)=0$ while $x_1\cdots x_N\ne 0$ and $g(x)\ne 0$.
\end{lem}
\begin{proof}
Assume, to the contrary, that every solution of $f(x)=0$ satisfies $x_1\cdots x_N=0$ or $g(x)=0$. Then the polynomial $x_1\cdots x_Ng(x)$ vanishes on the algebraic set
	\[
	\{x\in \C^N: f(x)=0\}.
	\]
Hilbert's Nullstellensatz implies the existence of a positive integer $r$ such that $f(x)$ divides $(x_1\cdots x_N)^rg(x)^r$. Since $f$ and $g$ are relatively prime and $\C[x_1,\ldots x_N]$ is a UFD, it follows that $f(x)$ must divide $(x_1\cdots x_N)^r$. However, this can only happen if $f$ is a monomial. If $f$ is not a monomial, then we conclude that there must be a solution of $f(x)=0$ for which $x_1\cdots x_N\ne 0$ and $g(x)\ne 0$, as claimed.
\end{proof}

Lastly we prove the analogous extension of \cite[Lemma~2.5]{BT}.
\begin{lem}\label{lem: bt 2}
Let $f,g\in \C[x_1,\ldots,x_N]$, neither identically $0$, and suppose that $y\in \C^N$ is a zero of $f$ but not $g$. Then, for any $\varepsilon>0$, there exists $\delta>0$ such that, for any polynomials $\tilde{f},\tilde{g}\in \C[x_1,\ldots,x_N]$ obtained by changing any of the nonzero coefficients of $f,g$ by at most $\delta$ each, there is a $z\in \C^N$ with $|y-z|<\varepsilon$ such that $\tilde{f}(z)=0$ but $\tilde{g}(z)\ne 0$.
\end{lem}
\begin{proof}
	Take $u$ with $|u|=1$ so that neither $f(y+tu)$ nor $g(y+tu)$ vanishes for all complex $t$. We choose $\varepsilon'$ with $0<\varepsilon'<\varepsilon$ such that $f(y+tu)$ does not vanish on the circle $C_{\varepsilon'}=\{t\in \C:|t|=\varepsilon'\}$. We take $\gamma_1>0$ smaller than the minimum of $f(y+tu)$ on $C_{\varepsilon'}$ and such that $g(y+tu)$ does not vanish on the disk $|t|\le \gamma_1$; we also set $\gamma_2$ as the minimum of $|g(y+tu)|$ for $t\in C_{\varepsilon'}$. On $C_{\varepsilon'}$ we have
	\[
	|f(y+tu)-\tilde{f}(y+tu)|<\delta N_1(1+\varepsilon+|y|)^{\deg f},
	\]
	where $N_1$ is the number of nonzero coefficients of $f$. Similarly, we also have
	\[
	|g(y+tu)-\tilde{g}(y+tu)|<\delta N_2(1+\varepsilon+|y|)^{\deg g}.
	\]
	We take $\delta$ small enough to make the right-hand side of each of the displays above smaller than $\min\{\gamma_1,\gamma_2\}$. Thus we have
	\[
	|f(y+tu)-\tilde{f}(y+tu)|<|f(y+tu)|
	\]
	as well as
	\[
	|g(y+tu)-\tilde{g}(y+tu)|<|g(y+tu)|
	\]
	on $C_{\varepsilon'}$. Rouch\'e's Theorem then implies that $\tilde{f}$ has a zero $z$ satisfying $|y-z|<\varepsilon'<\varepsilon$, while $\tilde{g}$ is nonvanishing on this same disk.
\end{proof}

\section{Proof of Theorem~\ref{thm: >1}}
Before starting the proof, we introduce one last piece of notation for the sake of brevity. Given a polynomial $f\in \cF[x_1,\ldots,x_N]$,
    \[
    f(x_1,\ldots,x_N)=\sum_{i=1}^M\left(\sum_{n\in\langle \cQ_i\rangle}\frac{a_i(n)}{n^s}\right)\prod_{j=1}^Nx_j^{\alpha_{i,j}},
    \]
we write $p_f$ to denote the largest prime in the union $\cup_i\cQ_i$, so that each $\cQ_i$ is contained in the interval $[2,p_f]$

Now we proceed with the proof, which closely follows that of \cite[Theorem~3]{R}. Suppose, after relabeling indices if necessary, that $F_1=G_1,\ldots,F_M=G_M$ for some $M\le \min\{N_f,N_g\}$, but $F_i\ne G_j$ for any $i$ and $j$ whenever $M<i\le N_f$ and $M<j\le N_g$. We set $N:=N_f+N_g-M$ and consider $f$ and $g$ as polynomials in the variables $X_1,\ldots,X_N$. Specifically, we order our variables so that $f$ depends only on $X_1,\ldots,X_{N_f}$, while $g$ depends only on $X_1,\ldots,X_M$ and $X_{N_f+1},\ldots,X_N$. With this in mind, we define auxiliary functions
	\[
	\tilde{f}(X_1,\ldots,X_N;s)=f\left(X_1\prod_{p\le p_{fg}}F_{1,p}(s),\ldots,X_{N_f}\prod_{p\le p_{fg}}F_{N_f,p}(s),\smash[b]{\!\underbrace{0,\ldots,0}_\text{$N-N_f$ times}}\right),
	\]
    \begin{multline*}
    \tilde{g}(X_1,\ldots,X_N;s)=g\left(X_1\prod_{p\le p_{fg}}F_{1,p}(s),\ldots,X_M\prod_{p\le p_{fg}}F_{M,p}(s),\smash[b]{\!\underbrace{0,\ldots,0,}_\text{$N-N_g$ times}}\right. \\
    \left. X_{N_f+1}\prod_{p\le p_{fg}}G_{M+1,p}(s),\ldots,X_{N_g}\prod_{p\le p_{fg}}G_{N_g,p}(s)\right).
    \end{multline*}
Clearly $\tilde{f}(s)$ and $\tilde{g}(s)$ are both polynomials in $\cF[X_1,\ldots,X_N]$ for any $s$ with $\sigma> 1$. Moreover, the coefficients of $\tilde{f}$ and $\tilde{g}$ are holomorphic in the open half-plane $\sigma>1$ and extend continuously to the line $\sigma=1$. Hence, by the maximum modulus principle, there is a real number $t_0$ such that these coefficients are all non-vanishing for $s=1+it_0$. Thus, Lemma~\ref{lem: bt 1} implies the existence of non-zero complex numbers $x_1,\ldots,x_N$ such that
    \[
    \tilde{f}(x_1,\ldots,x_N;1+it_0)=0,
    \]
    \[
    \tilde{g}(x_1,\ldots,x_N;1+it_0)\ne 0
    \]
simultaneously hold. Now take $R\ge 2$ small enough so that $\frac{2}{R}\le |x_n|\le \frac{R}{2}$ for all $1\le n\le N$. We then apply Lemma~\ref{lem: bt 2} with $\varepsilon=\frac{1}{R}$. Doing so, we are guaranteed the existence of an $\eta>0$ such that, for any $1<\sigma\le 1+\eta$, there is a $z(\sigma)\in \C^N$ with the following properties: $\frac{1}{R}\le |z_n(\sigma)|\le R$ for each $1\le n\le N$, and $\tilde{f}(z(\sigma);\sigma+it_0)=0$ but $\tilde{g}(z(\sigma);\sigma+it_0)\ne 0$. By Lemma~\ref{prop: R} with $y=p_{fg}$, we can find real numbers $t_p$ so that
    \[
    \begin{alignedat}{2}
    \prod_{p>p_{fg}}F_{n,p}(\sigma+it_p)&=z_n(\sigma)  &&\hbox{for } 1\le n\le N_f,\\
    \prod_{p>p_{fg}}G_{M+n,p}(\sigma+it_p)&=z_{N_f+n}(\sigma)\qquad &&\hbox{for } 1\le n\le N_g-M.
    \end{alignedat}
    \]
We set $t_p=t_0$ for the remaining primes $p\le p_{fg}$, and so it follows that
    \begin{align*}
    f\left(\prod_pF_{1,p}(\sigma+t_p),\ldots,\prod_pF_{N_f,p}(\sigma+t_p)\right) &= 0,\\
    g\left(\prod_pG_{1,p}(\sigma+t_p),\ldots,\prod_pG_{N_g,p}(\sigma+t_p)\right) &\ne 0.
    \end{align*}
The result then follows from a standard argument using the almost-periodicity of absolutely convergent Dirichlet series.$\hfill\square$

\section{Acknowledgements}
I would like to thank Steve Gonek for suggesting the problem that inspired this paper, a topic which comprised part of my Ph.D. thesis. I am also grateful for the valuable input and advice received from anonymous referees on a preliminary draft. This work was completed during a postdoctoral fellowship at the University of Exeter supported by the Leverhulme Trust (RPG-2017-
320), under the research grant `Moments of
$L$-functions in Function Fields and Random Matrix Theory.'

\bibliographystyle{abbrv}
\bibliography{refs}	

\end{document}